\documentclass[reqno,11pt]{amsart}

\usepackage{amsmath,amssymb,amsthm,tikz,mathtools}
\usepackage{verbatim}
\usepackage{xcolor}
\usepackage{mathrsfs}
\usepackage{cancel}
\usepackage{hyperref}
\usepackage[normalem]{ulem}

\usepackage{centernot}
\usepackage{mathtools}

\usepackage{tikz-cd}
\usetikzlibrary{circuits.logic.US}
\usetikzlibrary{shapes.geometric}
\usetikzlibrary{graphs, arrows.meta, positioning}
\usepackage{pgfplots}

\newtheorem{theorem}{Theorem}[section]

\newtheorem{lemma}{Lemma}[section]
\newtheorem{corollary}{Corollary}[section]
\newtheorem{proposition}{Proposition}[section]

\numberwithin{equation}{section} 

\usepackage{setspace}
\numberwithin{equation}{section}
  
\title{On a nonlocal superconductivity problem}

\author[D.J.  Araújo]{Damião J. Araújo}
\address{UFPB, Department of Mathematics, Universidade Federal da Para\'iba, 58059-900, Jo\~ao Pessoa-PB, Brazil}{}
\email{araujo@mat.ufpb.br}

\author[A. Sobral]{Aelson Sobral}
\address{Applied Mathematics and Computational Sciences (AMCS), Computer, Electrical and Mathematical Sciences and Engineering Division (CEMSE), King Abdullah University of Science and Technology (KAUST), Thuwal, 23955-6900, Kingdom of Saudi Arabia}{}
\email{aelson.sobral@kaust.edu.sa}
 
\begin{document}





  
\begin{abstract} 
This paper investigates degenerate nonlocal free boundary problems arising in the context of superconductivity, extending the nonlocal counterpart to the work of Caffarelli, Salazar, and Shahgholian \cite{CS02, CSS04} in the local setting. In these models, no partial differential equation governs the moving sets where the gradient vanishes, meaning that test functions are only required to have a nonzero gradient. Our main results provide interior gradient H\"older regularity estimates for viscosity solutions.  
\end{abstract}
  
\date{\today}

\maketitle
 
\tableofcontents

\section{Introduction}\label{sct intro}

In this paper, we investigate regularity estimates of solutions to nonlocal free boundary problems, which emerge in the study of nonlocal superconductivity and degenerate diffusion problems. 

The significance of these models lies in their application to finance, particularly in scenarios involving jump processes. In this case, diffusion occurs at points where the cost function is not maximized, meaning that diffusion can only be inferred at non-critical points. Secondly, within the context of superconductivity, the problem serves as a nonlocal variant of the stationary equation in the mean-field model for superconducting vortices, see \cite{BR99, C95} and \cite{ESS}.

Local versions of this problem were explored by Caffarelli, Salazar, and Shahgholian in \cite{CS02, CSS04}, where they investigate fully nonlinear elliptic PDEs of the form
\begin{equation}\label{fbp2024}
F(x,D^2 u) = g(x,u) \quad \text{in } B_1 \cap \Omega,
\end{equation}
for $\Omega = \{|Du| \not = 0 \}$. See also \cite{FS}, for a broader class of free boundary problems. 

Building on the framework of the aforementioned works, we formulate the problem studied in this paper. Given a smooth function \( f \), we consider an unknown pair \( (u, \Omega) \), where \( u \) is a function defined in \( \mathbb{R}^n \) and \( \Omega \subset \mathbb{R}^n \) is an open set, satisfying in the viscosity sense
\begin{equation}\label{introprob}
\left\{
\begin{array}{rcl}
(-\Delta)^s u = f & \text{in} & B_1 \cap \Omega, \\
|Du| = 0 & \text{in} & B_1 \setminus \Omega.
\end{array}
\right.
\end{equation}  
Here, \( (-\Delta)^s \) denotes the fractional Laplacian (see Section \ref{prelim-section} for details). Problem \eqref{introprob} constitutes a genuine free boundary problem, where the nonlocal diffusion properties break down near \( \partial \Omega \), while the complementary set retains a local character, marked by the presence of critical points.

\subsection{Challenges for the nonlocal setting} The nonlocal setting involves specific features that do not arise in the local framework, even in the simplest case $f=0$, as critical points have a stronger influence on the system due to the nonlocal nature of the problem. 

In the local setting, the breakthrough work of Imbert and Silvestre \cite{IS13} shows that functions which are harmonic on the set of non-critical points are, in fact, harmonic throughout the entire domain. A slightly more technical statement is as follows: given an open set $\Omega \subset \mathbb{R}^n$, such that $\{ |Du| \neq 0 \} \subseteq \Omega$, the equation
\begin{equation}\label{Oharmonic}
\Delta u = 0 \quad \mbox{in} \quad \Omega \cap B_1,    
\end{equation}
holds, if and only if, $u$ satisfies $\Delta u = 0$ in $B_1$.

However, in the nonlocal scenario, significant challenges arise, and such equivalence generally fails to hold. In fact, for the homogeneous problem
\begin{equation}\label{mainprob}  
\left\{  
\begin{array}{rcl}  
(-\Delta)^s u = 0 & \text{in} & B_1 \cap \Omega, \\  
|Du| = 0 & \text{in} & B_1 \setminus \Omega,  
\end{array}  
\right.  
\end{equation}  
one might be tempted to infer that solutions of \eqref{mainprob} are $s$-harmonic in $B_1$; however, this is not true in general. Indeed, explicit counterexamples show that such an equivalence fails in the nonlocal setting, see, for instance, \cite{APT, PT}. The reverse implication is trivial in both cases. In light of this, we observe that problem \eqref{introprob} remains significant even in the homogeneous case, further motivating our investigation of the regularity estimates of solutions in the nonlocal setting.

One of the key components of the approach adopted by Caffarelli and Salazar in the local case \cite{CS02} is that solutions to problem \eqref{fbp2024} actually satisfy a uniformly elliptic PDE with a bounded right-hand side throughout the entire domain. However, in a nonlocal setting, such a reduction is generally not available. The nature of the fractional Laplacian, requires integration throughout the space $\mathbb{R}^n$, and thus the behavior of the solution outside the domain $B_1$ can heavily influence, as solutions are only $s$-integrable in $\mathbb{R}^n \setminus B_1$, any attempt to extend the PDE globally in $B_1$ would introduce irregular or singular terms.

\subsection{Main results and consequences}
Our main goal is to develop a De Giorgi-type gradient oscillation method tailored to the nonlocal setting, drawing inspiration from elliptic degeneracy scenarios as studied in \cite{IJS, JS}.

We shall consider solutions $(u, \Omega)$ of problem \eqref{mainprob}, using this problem as our primary prototype throughout the paper -- further discussions for broader settings, shall be discussed in Section \ref{extensions-sct}. 

Here is our main result. 

\begin{theorem}\label{main theorem}
For $u \in C(B_1) \cap L^\infty(\mathbb{R}^n)$ and $\Omega \subset \mathbb{R}^n$ an open set, assume that $(u,\Omega)$ solves \eqref{mainprob}, for some $s\in (1/2,1)$. Then, $u$ is locally $C^{1,\alpha}$, for some universal $\alpha\in (0,1)$, depending only on $n$ and $s$. Furthermore, there exists $C$ depending on $n$ and $s$, such that
$$
    \|u\|_{C^{1,\alpha}(B_{1/2})} \leq C\,\|u\|_{L^\infty(\mathbb{R}^n)}.
$$   
\end{theorem}

The strategy relies on a positivity argument applied to the derivatives, coupled with a discrete normalization scheme that controls the nonvanishing derivatives. At each iteration step, nonlocal contributions are carefully handled to ensure that the structure of the problem is preserved throughout the process. In this context, if for some direction the gradient is large in measure at a given step, we consider the PDE region to apply an appropriate rescaling combined with a small perturbation argument, leading to an improved regularity (Proposition \ref{small pert argument}). Conversely, if the gradient remains small in measure and in every direction for an infinite number of steps, this reveals the presence of critical points, allowing us to iteratively refine the oscillation (Proposition \ref{reg-at-critical-points}), overcoming the absence of PDE information at these points.

In our next result, we derive similar results to those obtained for the problem \eqref{introprob}, provided that $f$ is sufficiently smooth. Moreover, the estimate in Theorem \ref{main theorem} is refined by replacing the global $L^\infty$ bound for $u$ in $\mathbb{R}^n$ with the weaker and more natural $L^1_s$-norm, better suited to the nonlocal context.

\begin{theorem}\label{general theorem}
For $u \in C(B_1) \cap L^1_s(\mathbb{R}^n)$ and $\Omega \subset \mathbb{R}^n$ an open set, assume that $(u,\Omega)$ solves \eqref{introprob}, for some $s\in (1/2,1)$. Then, $u$ is locally $C^{1,\alpha}$, for some universal $\alpha \in (0,1)$, depending only on $n$ and $s$. Furthermore, there exists $C$ depending on $n$ and $s$, such that
$$
    \|u\|_{C^{1,\alpha}(B_{1/2})} \leq C\left(\|u\|_{L^\infty(B_1)} +  \|u\|_{L^1_s(\mathbb{R}^n)} + \|f\|_{C^{0,1}(B_1)} \right).
$$
\end{theorem}

Since our methods are purely nonlinear and naturally extend to a broader class of nonlocal operators, we dedicate Section \ref{extensions-sct} to focus on a class of fully nonlinear operators. Moreover, we emphasize that our results remain stable as \( s \to 1 \), seamlessly recovering the local framework classically studied in \cite{CS02} and \cite{CSS04}.

The regularity assumption on $f$ is natural in view of the structure of the argument. Indeed, at a certain stage, $L^{\varepsilon}$-estimates must be applied, which are currently available only under uniform $L^\infty$ bounds on $\partial_i f$, since the PDE is differentiated in the region $B_1 \cap \Omega$. We nevertheless believe that this condition could be weakened. However, the primary focus of the present paper is the development of sharp estimates in the homogeneous setting, which we expect to constitute a key ingredient for extending the theory to weaker regularity assumptions on $f$. Establishing such extensions is left for future investigation.

Obstacle-type problems of the form
\begin{equation}\label{saojoao}
\min\{-(-\Delta)^s v,\, v - \varphi \} = 0 \quad \text{in} \quad B_1,
\end{equation}
has been investigated in \cite{ARO20, BFRO18, JN17, CROS17}. In contrast with this class of free boundary problems, problem \eqref{introprob} presents additional difficulties: no supersolution condition is imposed throughout the entire domain $B_1$, and no lower bound is enforced by an obstacle. From this perspective, we observe that problem \eqref{saojoao} can be considered in the following class of problems
\begin{equation}\label{saopedro}
(-\Delta)^s u = f \quad \text{in } B_1 \cap \Omega, \qquad
\Omega \supseteq \{ Du \neq D\varphi \}.
\end{equation}
In light of this, from Theorem \ref{general theorem} we derive the following consequence.

\begin{corollary}
Let $s \in (1/2,1)$, and assume that $u \in C(B_1) \cap L^1_s(\mathbb{R}^n)$ solves \eqref{saopedro}. Then, there exist constants $\alpha \in (0,1)$ and $C > 0$, depending only on $n$ and $s$, such that
$$
\|u\|_{C^{1,\alpha}(B_{1/2})} \leq C(\|u\|_{L^\infty(B_1)} +  \|u\|_{L^1_s(\mathbb{R}^n)} + \|f\|_{C^{0,1}(B_1)} + \|(-\Delta)^s \varphi\|_{C^{0,1}(B_1)}).
$$
\end{corollary}

Theorem \ref{main theorem} further allows us to observe that the derivatives of solutions to \eqref{mainprob} are viscosity solutions within the framework developed by Ros-Oton and Serra \cite{ROS17}, see also \cite{ARO20, RO16}. From this, assuming that $\Omega$ is a $C^{1,\mu}$ domain, optimal $C^{1,s}$ regularity for solutions can be established.

The organization of the paper goes as follows. In Section \ref{prelim-section}, we establish basic definitions and known results. In Section \ref{pos-arg-section}, we provide the positive argument for derivatives. In Section \ref{smallperturb-section}, we discuss a small perturbation approach. We conclude the proof of the main results in Section \ref{proof-section}, and provide further extensions in Section \ref{extensions-sct}.

\section{Preliminaries}\label{prelim-section}

In this section, we introduce the basic definitions and results to be used throughout the paper, along with an auxiliary problem that shall be crucial for the analysis. 

To address problem \eqref{mainprob}, we begin by presenting concepts related to $s$-harmonic functions and their equivalences. A function \( u \), sufficiently smooth and defined in \(\mathbb{R}^n\), is called a \( s \)-harmonic function in a domain \(\mathcal{O} \subset \mathbb{R}^n\) if
\begin{equation}\label{frac lap as pseudo}
    (-\Delta)^s u(x) \coloneqq C_{n,s} \lim_{\epsilon \to 0^+}\int_{\mathbb{R}^n \backslash B_\epsilon(x)} \frac{u(x) - u(y)}{|x-y|^{n+2s}}\,dy = 0,
\end{equation}
for every \( x \in \mathcal{O} \), where $C_{n,s}$ is a normalizing constant depending on dimension and $s$, see \cite{DPV}. For the sake of simplicity, we will adopt the following notation $\Delta^s \coloneqq -(-\Delta)^s$. The class of \(s\)-harmonic functions arises, for instance, as the Euler-Lagrange equations associated with minimizers of the functional
\[
[u]_{H^s(\mathbb{R}^n)}^2 \coloneqq \iint_{\mathbb{R}^n \times \mathbb{R}^n}\frac{|u(x) - u(y)|^2}{|y-x|^{n+2s}}\,dx\,dy \longrightarrow \min, \quad u \in H^s(\mathbb{R}^n),
\]
see \cite{SILV06} for details. 

We remark that, for the concept of $s$-harmonic functions in the context of viscosity solutions, the most convenient way to define the fractional Laplacian is through formula \eqref{frac lap as pseudo}, where we replace $u$ by $C^2$ touching functions near the integral singularity; see \cite[Definition 2.2]{CS09}. For a further discussion, we refer to \cite{BI08}. 

To address gradient regularity for solutions to nonlocal equations, it is essential to establish a convenient prescription for growth at infinity. Thus, we define the space of functions whose growth is appropriately controlled. Specifically, we say that a function \(u \colon \mathbb{R}^n \to \mathbb{R}\) is in \(L^1_s(\mathbb{R}^n)\) if it satisfies
\[
\|u\|_{L^1_s(\mathbb{R}^n)} \coloneqq \int_{\mathbb{R}^n} \frac{|u(y)|}{1 + |y|^{n + 2s}} \, dy < \infty.
\]

Next, we state Lipschitz estimates of solutions to \eqref{mainprob}. The proof is a careful adaptation of the Ishii-Lions method (see \cite{BCI08, BI08, CIL92}), and follows from the fact that solutions of \eqref{mainprob} solve in particular equation $|Du|\Delta^s u = 0$. For a detailed proof, we refer to \cite{PT}.

\begin{proposition}\label{LIPS-VEI-DE-GUERRA}
Let $u \in L^1_s(\mathbb{R}^n)$ be a solution to \eqref{mainprob}. Then, there is a constant $C$ depending on dimension and $s$ such that
$$
    \|Du\|_{L^\infty(B_{3/4})} \leq C(\|u\|_{L^\infty(B_1)} + \|u\|_{L^1_s(\mathbb{R}^n)}).
$$
\end{proposition}

Next, we study an auxiliary problem that will be required to obtain gradient oscillation estimates. For \(\xi \in \mathbb{S}^{n-1}\) and $\nu \in \mathbb{R}$, we consider the following problem
\begin{equation}\label{shifteq}
    \Delta^s u = 0 \quad \text{in} \quad \{ |\nu \, Du + \xi| \neq 0 \} \cap B_1.
\end{equation}
Solutions are understood in the viscosity sense, for test functions $\varphi$ that touch $u$ at point $x$, satisfying $|\nu D\varphi(x)+\xi|>0$, see \cite[Defintion 1.3]{PT}. Additionally, we observe that solutions for \eqref{shifteq} are expected to be in $L^1_s(\mathbb{R}^n)$, as they further satisfy inequality 
\begin{equation}\label{inftycond}
|u(x)| \leq \max\left\{1,|x|^{1+\alpha_1} \right\} \quad \mbox{in} \quad \mathbb{R}^n,
\end{equation}
for some $0 < \alpha_1 < 2s-1$. In particular, condition \eqref{inftycond} implies 
\begin{equation}\label{twohours}
\begin{array}{ccl}
    \|u\|_{L^1_s(\mathbb{R}^n)} & = &  \displaystyle \int_{\mathbb{R}^n}\frac{|u(y)|}{1+|y|^{n+2s}}dy\\ 
    & \leq & |B_1| +  \displaystyle \int_{\mathbb{R}^n \backslash B_1}\frac{|y|^{1+\alpha_1}}{|y|^{n+2s}}dy = |B_1| + \frac{|\mathbb{S}^{n-1}|}{2s - 1 - \alpha_1}.
\end{array}
\end{equation}

Lipschitz estimates for solutions of \eqref{shifteq} are available, with the advantage that they are uniform with respect to \(\xi \in \mathbb{S}^{n-1}\) (see \cite[Lemma 2.2]{PT}). For the reader's convenience, we present this result as follows.

\begin{proposition}\label{ISHII-LIONS-QUE-N-VAMOS-FAZER}
Let $u$ be a solution to \eqref{shifteq}, for $\xi \in \mathbb{S}^{n-1}$. If 
$$
|u(x)| \leq \max\left\{1,|x|^{1+\alpha_1}\right\} \quad \mbox{in } \; \mathbb{R}^n,
$$
then there is a constant $\Lambda = \Lambda(n,s)$ such that
$$
\|Du\|_{L^\infty(B_{3/4})} \leq \Lambda,
$$
provided $\nu$ is universally small enough.
\end{proposition}

We conclude this section by recording the following $L^{\epsilon}$-estimate, which is stated to the fractional Laplacian setting in \cite[Theorem 10.4]{CS09} and provides a quantitative control on the measure of sublevel sets.
\begin{lemma}\label{lemma:CS09}
Let $r>0$ and $z \in \mathbb{R}^n$. Consider $u \geq 0$ in $\mathbb{R}^n$ and $\Delta^s u \leq C_0$ in $B_{2r}(z)$. Then, there are universal constants $C>0$ and $\epsilon>0$ such that
$$
    |\{u<t\}\cap B_r(z)| \leq Cr^n(u(z) + C_0r^{2s})^\epsilon t^{-\epsilon} \quad \forall t>0.
$$
\end{lemma}

\section{Nonlocal positivity argument}\label{pos-arg-section}

In this section, we develop a positivity argument to improve the oscillation of $Du$ in dyadic balls. The idea is that if the derivative is small in a region of positive measure, then it oscillates in a controlled fashion in a smaller portion of the region. In what follows, we provide a subsolution information for derivatives of a given solution. 

\begin{lemma}\label{subsolution-lemma}
Let $\eta \colon \mathbb{R}^n \to [0,1]$ be a smooth cut-off function satisfying
$$
    \eta = 1 \quad \mbox{in} \quad  B_1, \quad \mbox{and} \quad \eta = 0 \quad \mbox{in} \quad \mathbb{R}^n \backslash B_2.
$$
If $u$ is a solution to \eqref{mainprob} and $e \in \mathbb{S}^{n-1}$, then $v = \eta\,(\partial_e u - \mu)_+$ solves
$$
    \Delta^s v \geq -C\|u\|_{L^1_s(\mathbb{R}^n)} \quad \mbox{in }\; B_{1/2},
$$
for any $\mu \in (0,1)$ and $C$ is a dimensional constant.
\end{lemma}
\begin{proof}
Recall that $u$ is $s$-harmonic in $\Omega \cap B_1$ and thus smooth in this region. First, let us consider $x \in \Omega \cap B_1$. We proceed through a cut-off argument. Let $e \in \mathbb{S}^{n-1}$ and define
$$
    w^h(z) \coloneqq \frac{u(z+he) - u(z)}{h}.
$$
For $h$ small enough, it follows that $x+he \in \Omega \cap B_1$. Consequently, we have $\Delta^s w^h(x) = 0$. Now, for $\eta$ as in the statement of Lemma \ref{subsolution-lemma}, we write
$$
    w^h = \eta\,w^h + (1-\eta)\,w^h \eqqcolon w_1 + w_2.
$$
Hence, we obtain
$$
    \Delta^s w_1(x) = f(x), \quad \mbox{for} \quad f \coloneqq -\Delta^s w_2.
$$
Using change of variables, we obtain
\begin{eqnarray*}
    f(x) & = & \int_{\mathbb{R}^n}\frac{w_2(y)}{|y-x|^{n+2s}}dy\\
        & = & \int_{\mathbb{R}^n}(1-\eta(y))\frac{(u(y+he)-u(y))}{h}\frac{1}{|y-x|^{n+2s}}dy\\
        & = & \int_{\mathbb{R}^n}u(y)\left(\frac{g(y-he)-g(y)}{h} \right)dy,
\end{eqnarray*}
where $g(z) \coloneqq (1-\eta(z))|z-x|^{-(n+2s)}$. By the mean value theorem, we have $|g(y-he)-g(y)| \leq |Dg(y + \theta e)|h$, for $\theta \in (0,h)$. By direct computations, we observe that
\begin{eqnarray*}
    |Dg(z)| & \leq & |D\eta(z)|\,|z-x|^{-(n+2s)} + (n+2s)(1-\eta(z))|z-x|^{-(n+1+2s)}\\
            & \leq & C\,|z-x|^{-(n+2s)}\chi_{B_2\backslash B_1} + (n+2)|z-x|^{-(n+1+2s)}\chi_{\mathbb{R}^n \setminus B_1}.
\end{eqnarray*}
Since $x \in B_{1/2}$ and $z \in \mathbb{R}^n \backslash B_1$, we have $|z-x| \geq \frac{1}{2}|z|$, and so
$$
    |Dg(z)| \leq C|z|^{-(n+2s)}\chi_{\mathbb{R}^n \backslash B_1}.
$$
As a consequence,
$$
    |f(x)| \leq C\int_{\mathbb{R}^n \backslash B_{1/2}} u(z) \frac{1}{|z|^{n+2s}} \leq C(n)\,\|u\|_{L^1_s(\mathbb{R}^n)}.
$$
Thus, we have obtained
$$
    \left|\Delta^s w_1 \right| \leq C(n)\|u\|_{L^1_s(\mathbb{R}^n)} \quad \mbox{in} \quad \Omega \cap B_1,
$$
where $w_1 = \eta w^h$. Passing to the limit as $h \to 0$ we obtain
$$
    \left| \Delta^s (\eta \, \partial_e u)\right| \leq C\,\|u\|_{L^1_s(\mathbb{R}^n)} \quad \mbox{in} \quad B_1 \cap \Omega.
$$
Now observe that if $v = \eta\,(\partial_e u - \mu)_+$ and $x \in \{v>0 \}\cap B_1$, then
\begin{eqnarray*}
    \Delta^s v(x) & = & \int_{\mathbb{R}^n}\frac{(\partial_e u(y) - \mu)_+ - (\partial_e u(x) - \mu)}{|y-x|^{n+2s}}dy\\
    & \geq & \int_{\mathbb{R}^n}\frac{\eta(y)\,\partial_e u(y) - \partial_e u(x) + (1-\eta(y))\mu}{|y-x|^{n+2s}}dy.
\end{eqnarray*}
Since $1-\eta \geq 0$, it then follows that
$$
    \Delta^s v(x) \geq \Delta^s (\eta \partial_e u)(x) \geq - C\|u\|_{L^1_s(\mathbb{R}^n)},
$$
where we have also used that since $x \in \{v>0 \}$, then $\partial_e u(x) \not = 0$. As a consequence, it follows that $Du(x) \not = 0$ and so $x \in \Omega \cap B_1$. If $x \in \{v=0 \}$, then it is straightforward to see that
$$
    \Delta^s v(x) = \int_{\mathbb{R}^n}\frac{v(y)}{|y-x|^{n+2s}}dy \geq 0. 
$$
In any case, it holds
$$
    \Delta^s v \geq -C\|u\|_{L^1_s(\mathbb{R}^n)} \quad \mbox{in }\; B_{1/2}.
$$
\end{proof}

Now, we present the gradient improvement of oscillation estimates. For simplicity, we assume throughout this section that \(u\) is a solution to \eqref{mainprob} satisfying \(u(0) = 0\). 

\begin{lemma}\label{reg-at-critical-points-step1} 
Assume $u$ satisfies $\|Du\|_{L^\infty(B_{1/2})} \leq 1$ and
$$
    |u(x)|\leq \max\{1,|x|^{1+\alpha_1}\} \quad \mbox{in} \quad \mathbb{R}^n.
$$
Given $\mu, \delta \in (0,1)$, there exist positive parameters $\mu_\star$ and $r_\star$ depending only on $n$, $s$, $\mu$ and $\delta$, such that the following holds: for a given $e \in \mathbb{S}^{n-1}$, if
$$
    \left|\left\{x \in B_{r_\star} \colon Du(x) \cdot e \leq  \delta \right\} \right| > \mu |B_{r_\star}|,
$$
then
$$
    Du \cdot e \leq 1-\mu_\star \quad \mbox{in} \quad B_{r_\star/4}.
$$
\end{lemma}
\begin{proof}
Let \(\eta\) be as specified in the assumptions of Lemma \ref{subsolution-lemma}. Defining \( w \coloneqq (Du \cdot e - \delta)_+ \), we apply the latter Lemma and conclude that \( w \) is a solution to
$$
    \Delta^s \left(\eta w\right) \geq -C_1\|u\|_{L^1_s(\mathbb{R}^n)} \geq -C_0  \quad \mbox{in} \quad B_{1/2}.
$$
In the latter inequality, we follow \eqref{twohours}, for $C_0$ a constant depending only on $n$ and $s$. 
 
Thanks to $\|Du\|_{L^\infty(B_{1/2})} \leq 1$, we observe that $w \leq 1 - \delta$ in $B_{1/2}$. This implies that function $\overline{w} \coloneqq (1-\delta - w)_+$ satisfies
$$
     \Delta^s \left(\eta \overline{w}\right) \leq C_0 \quad \mbox{in} \quad B_{1/2}.
$$
Moreover, 
$$
    \left|\left\{x \in B_{r_\star}\colon \overline{w} \geq 1-\delta   \right\}\right| > \mu\left|B_{r_\star}\right|.
$$
We now observe that for any $y \in B_{r_\star/4}$, the inclusion $B_{r_\star} \subset B_{3r_\star/2}(y)$ holds. Consequently, Lemma~\ref{lemma:CS09} can be applied to yield
\begin{eqnarray*}
    \mu\,|B_{r_\star}| &< & \left|\left\{x \in B_{r_\star}\colon \overline{w} \geq 1-\delta   \right\}\right|\\
    & \leq & \left|\left\{x \in B_{3r_\star/2}(y) \colon \overline{w} \geq 1-\delta   \right\}\right|\\
    & \leq & Cr_{\star}^n(\overline{w}(x) + C_0r_{\star}^{2s})^\epsilon\, (1-\delta)^{-\epsilon},
\end{eqnarray*}
for $y \in B_{r_\star/4}$. Rearranging terms,
$$
    c(1-\delta)\,\mu^{\frac{1}{\epsilon}} \leq \overline{w}(y) + r_\star^{2s} C_0,
$$
for some $c>0$. Next, we take $r_\star$ small depending on $\delta$, $\mu$, $c$ and $C_0$, so that
$$
    \frac{c(1-\delta)\,\mu^{\frac{1}{\epsilon}}}{2} \leq \overline{w} \quad B_{r_\star/4}.
$$
Finally, by the definition of $\overline{w}$, we conclude that
$$
    Du\cdot e \leq  1-\mu_\star \quad \mbox{in} \quad B_{r_\star/4},
$$
for $\mu_\star \coloneqq 2^{-1}c\mu^{\frac{1}{\epsilon}}(1-\delta)$.
\end{proof}

We now apply an iterative method to establish gradient control within dyadic balls. Unlike local cases, special care is needed to ensure that the growth at infinity for rescaled functions is maintained, which evidences the nonlocal influence in the argument. For notational simplicity, let us define \( I_k \coloneqq \{0, 1, \dots, k\} \).

\begin{proposition}\label{reg-at-critical-points} 
Assume $u$ satisfies $\|Du\|_{L^\infty(B_{1/2})} \leq 1$ and
$$
    |u(x)|\leq \max\{1,|x|^{1+\alpha_1}\} \quad \mbox{in} \quad \mathbb{R}^n.
$$
Given $\mu, \delta \in (0,1)$, there exist positive parameters $r_\star$, $\lambda$ and $\alpha$ depending only on $n$, $s$, $\mu$ and $\delta$, such that the following holds: given $k>0$ integer, assume that
\begin{equation}\label{positivity every direction}
    \inf_{e \in \mathbb{S}^{n-1}}\left|\left\{x \in B_{r_\star\lambda^i}\colon Du(x) \cdot e \leq \delta \lambda^{\alpha i}  \right\} \right| > \mu |B_{r_\star \lambda^i}|,
\end{equation}
for each $i \in I_k$, then
\begin{equation}\label{boladouro}  \|Du\|_{L^\infty(B_{r_\star \lambda^{i+1}})} \leq \lambda^{\alpha(i+1)}.
\end{equation}
\end{proposition}

\begin{proof}
Initially, we consider $\lambda>0$ small enough satisfying
\begin{equation}\label{smallness regime of lambda in reg at criti points}
    \max\left(\lambda\, \frac{2}{r_\star}, \lambda^{\alpha_1/2}\left(\frac{4}{r_\star}\right)^{1+\alpha_1}\right) \leq 1.
\end{equation}
For $\mu_\star$ as in Lemma \ref{reg-at-critical-points-step1}, we consider 
$$
    v_{i+1}(x) \coloneqq \frac{v_i(\lambda x)}{\lambda (1-\mu_\star)},
$$
for each $i \in I_k$, where $v_0 = u$.  

Next, we claim that if $v_j(0) = 0$, $\|Dv_j\|_{L^\infty(B_{1/2})} \leq 1$, and 
\begin{equation}\label{growth for v_i}
|v_j(x)|\leq \max\{1,|x|^{1+\alpha_1}\} \quad \mbox{in} \quad \mathbb{R}^n, 
\end{equation}
holds for $j=i \leq k$, then the same holds for $j=i+1$. Indeed, we easily have $v_{i+1}(0) = 0$. Considering 
\begin{equation}\label{triestekaust}
\alpha \coloneqq \frac{\ln(1-\mu_\star)}{\ln(\lambda)},
\end{equation}
observe that 
$$
    v_i(x) = \frac{u(\lambda^i x)}{\lambda^{i(1+\alpha)}}.
$$
By further adjusting $\lambda$, we may assume $\alpha \leq \alpha_1/2$. Additionally, note that $v_i$ solves \eqref{mainprob}. By assumption \eqref{positivity every direction}, we notice that
$$
    \inf_{e \in \mathbb{S}^{n-1}}\left|\left\{x \in B_{r_\star}\colon Dv_i(x) \cdot e \leq \delta  \right\} \right| > \mu |B_{r_\star}|.
$$
From the choice in \eqref{triestekaust}, Lemma \ref{reg-at-critical-points-step1} applied to $v_i$ yields
\begin{equation}\label{osc_improv_for_v_i}
\|Dv_i\|_{L^\infty\left(B_{r_\star/4} \right)} \leq 1-\mu_\star = \lambda^\alpha. 
\end{equation}
Hence, we have that the estimate above implies
$$
\|Dv_{i+1}\|_{L^\infty\left(B_{1/2}\right)} \leq \|Dv_{i+1}\|_{L^\infty\left(B_{\lambda^{-1}\frac{r_\star }{4}}\right)} \leq 1.
$$
Finally, we shall conclude estimate \eqref{growth for v_i}. From the latter estimate, we use that $v_{i+1}(0)=0$ to obtain 
$$
    |v_{i+1}(x)| \leq \max\left\{1,|x|^{1+\alpha_1} \right\}, \quad \mbox{for} \quad x \in B_{\lambda^{-1}\frac{r_\star }{4}}.
$$
For the complementary set $\mathbb{R}^n \backslash B_{\lambda^{-1}\frac{r_\star }{4}}$, we use \eqref{growth for v_i} for $j=i$, to derive
\begin{eqnarray*}
|v_{i+1}(x)| & = & \lambda^{-(1+\alpha)}|v_i(\lambda x)|\\
& \leq & \lambda^{-(1+\alpha)}\max\left\{1, \lambda^{1+\alpha_1}|x|^{1+\alpha_1} \right\}\\
& = & \lambda^{\alpha_1 - \alpha}\max\left\{\lambda^{-(1+\alpha_1)} ,|x|^{1+\alpha_1} \right\}
\end{eqnarray*}
and so,
\begin{equation}\nonumber
|v_{i+1}(x)| \leq \lambda^{\alpha_1 - \alpha}\left(\frac{4}{r_\star}\right)^{1+\alpha_1} |x|^{1+\alpha_1}.
\end{equation}
Using that $\alpha_1 - \alpha \geq \alpha_1/2$, and inequality \eqref{smallness regime of lambda in reg at criti points}, we get 
$$
|v_{i+1}(x)| \leq \max\left\{1,|x|^{1+\alpha_1} \right\},
$$
as claimed.

Finally, we apply the claim recursively for each $i \in I_k$, where \eqref{osc_improv_for_v_i} gives
$$
    \|Du\|_{L^\infty\left(B_{r_\star \lambda^{i+1}}\right)} \leq \lambda^{\alpha(i+1)}.
$$
\end{proof}

\section{Oscillation estimates}\label{smallperturb-section}

In this section, we establish gradient oscillation estimates for solutions under small flatness assumptions. 

\begin{proposition}\label{small pert argument}
Let $u$ be a solution to \eqref{mainprob} satisfying
$$
    |u(x)| \leq \max\{1, |x|^{1+\alpha_1}\}, \quad \text{for} \quad x \in \mathbb{R}^n.
$$
There exist parameters $\lambda_\star$ and $C$, depending only on $n$ and $s$, such that the following holds: suppose there exists an affine function 
$\ell(x) \coloneqq a + \xi \cdot x$, with $\xi \in \mathbb{S}^{n-1}$, such that
\begin{equation}\label{proximity}
    \|u - \ell\|_{L^\infty(B_1)} \leq \lambda_\star^2.
\end{equation}
Then, we have
$$
    |Du(x) - Du(0)| \leq C |x|^{2s-1}, \quad \mbox{for} \quad x \in B_{1/2}.
$$
\end{proposition}

\begin{proof}
For $\beta < \alpha_1$, define
$$
    w(x) \coloneqq \frac{[u-\ell](\lambda_\star x)}{\lambda_\star^{1+\beta}}.
$$
for $\lambda_\star$ to be chosen later. First, we concentrate our analysis to show that 
$$
   |w(x)| \leq \max\left\{1,|x|^{1+\alpha_1}\right\} \quad \mbox{in} \quad \mathbb{R}^n. 
$$ 
Indeed, by \eqref{proximity} we easily have $|w(x)| \leq 1$ in $B_{\lambda_\star^{-1}}$. For $|x| > \lambda_\star^{-1}$, we obtain
\begin{eqnarray*}
    |w(x)| & \leq & \lambda_\star^{-(1+\beta)}\left(|u(\lambda_\star x)| + |a| + |\lambda_\star x|\right)\\
           & \leq & \lambda_\star^{-(1+\beta)}\left(\max\{1,\lambda_\star^{1+\alpha_1}|x|^{1+\alpha_1}\} + 2 + \lambda_\star|x|\right)\\
           & \leq & \lambda_\star^{\alpha_1 - \beta}\left(\max\left\{\lambda_\star^{-(1+\alpha_1)},|x|^{1+\alpha_1} \right\} + 2\lambda_\star^{-(1+\alpha_1)} + \lambda_\star^{-\alpha_1}|x|\right)\\
           & \leq & 4\lambda_\star^{\alpha_1 - \beta}|x|^{1+\alpha_1}.
\end{eqnarray*}
Assuming $4 \lambda_\star^{\alpha_1 - \beta} \leq 1$, it implies that $|w(x)| \leq |x|^{1+\alpha_1}$ for $x \in \mathbb{R}^n \backslash B_{ \lambda_\star^{-1}}$.

Secondly, we claim that
$$
   B_{3/4} \subseteq \{x \in B_{3/4} \colon \xi + \lambda_\star^\beta Dw(x) \not = 0 \}. 
$$
Indeed, since $w$ promptly satisfies \eqref{shifteq}, we take $\lambda_\star^\beta$ even smaller, and apply Proposition \ref{ISHII-LIONS-QUE-N-VAMOS-FAZER}, deriving 
$$
    \|Dw\|_{L^\infty(B_{3/4})} \leq \Lambda,
$$
for some universal $\Lambda>0$. Hence, we observe that
$$
    |\xi + \lambda_\star^\beta Dw(x)| \geq 1 - \lambda_\star^\beta \Lambda \geq \frac{1}{2},
$$
for each $x \in B_{3/4}$.

In view of this, $w$ is $s$-harmonic in $B_{3/4}$. By classical gradient regularity estimates
\begin{equation}\label{eq:estimate}
    |Dw(x) - Dw(0)| \leq C|x|^{2s-1} \quad \mbox{for} \quad x \in B_{1/2},
\end{equation}
for some $C>0$ depending only on $n$ and $s$. Scaling back $w$ to $u$, we have
$$
    |Du(x) - Du(0)| \leq C|x|^{2s-1} \quad \mbox{for} \quad x \in B_{\lambda_\star/2}.
$$
Finally, for $x \in B_{1/2} \backslash B_{\lambda_\star/2}$, we conclude that
\begin{eqnarray*}
    |Du(x) - Du(0)| \leq  \dfrac{4}{\lambda_\star}\|Du\|_{L^\infty(B_{1/2})}|x|^{2s-1}.
\end{eqnarray*}
\end{proof}

Before proceeding with the proof of the main theorem, we conclude this section by showing that, in a neighborhood of a nondegenerate point, a given Lipschitz function is close to an affine function with a unit slope. This result follows from a standard argument using Sobolev's inequality.

\begin{lemma}\label{proximity with affine functions}
Let $v \in W^{1,\infty}(B_2)$ such that $v(0)=0$ and $\|Dv\|_{L^\infty(B_2)} \leq 1$. Given $\varepsilon>0$, there exist $\mu$ and $\delta$ depending on $\varepsilon$ and $n$, such that if
$$
    \left|\left\{x \in B_{1} \colon Dv(x) \cdot e \leq \delta \right\} \right| \leq \mu |B_{1}|,
$$
for some $e \in \mathbb{S}^{n-1}$, then there exist $a \in [-1,1]$ and $\xi \in \mathbb{S}^{n-1}$, such that
$$
    \|v - \ell\|_{L^\infty(B_1)} \leq \varepsilon,
$$
for $\ell(x) \coloneqq a + \xi \cdot x$.
\end{lemma}

\begin{proof}
Let
$$
    a = \frac{1}{|B_1|}\int_{B_1}v.
$$
From Sobolev's inequality, 
\begin{eqnarray*}
    \left|v(x) - (a + e\cdot x)\right|^{2n} & \leq & C(n)\int_{B_1}|Dv(x) - e|^{2n}dx. 
\end{eqnarray*}
For simplicity, define 
$$
\mathcal{A} = \left\{x \in B_1 \colon Dv(x) \cdot e \leq \delta \right\}.
$$
By assumption, we have $|\mathcal{A}| \leq \mu |B_1|$, and so
\begin{eqnarray*}
    \int_{B_1}|Dv(x) - e|^{2n}dx & = & \int_{\mathcal{A} }|Dv(x) - e|^{2n}dx + \int_{B_1 \backslash \mathcal{A} }|Dv(x) - e|^{2n}dx\\
                    & \leq & 4^n|\mathcal{A}| + 4^n|B_1|(1-\delta)^{2n}\\
                    & \leq &  C(n)(\mu + (1-\delta)^{2n}) \leq \varepsilon^{2n},
\end{eqnarray*}
provided $\delta$ and $\mu$ are carefully chosen. Hence, it follows that
$$
    \|v - (a + e \cdot x)\|_{L^\infty(B_1)} \leq \varepsilon.
$$
In addition, since $v(0)=0$ and $\|Dv\|_{L^\infty(B_2)} \leq 1$, it implies that $a \in [-1,1]$.
\end{proof}

\section{Gradient regularity estimates}\label{proof-section}

In this section, we build on the results from Sections \ref{pos-arg-section} and \ref{smallperturb-section} to prove Theorem \ref{main theorem}, applying a dichotomy argument that considers possible degeneracy contexts. 

\begin{proof}[Proof of Theorem \ref{main theorem}]
For $K \coloneqq 2\|u\|_{L^\infty(\mathbb{R}^n)} + \|Du\|_{L^\infty(B_{3/4})}$ and $x_0 \in B_{1/2}$, denote
$$
    v(x) \coloneqq \frac{u(x_0+4^{-1}x) - u(x_0)}{K}.
$$
Let \(\lambda_\star\) be as defined in Proposition \ref{small pert argument}. Set \(\varepsilon = \lambda_\star^2\) in the assumptions of Lemma \ref{proximity with affine functions}, and let \(\mu\) and \(\delta\) be the corresponding parameters from that result. By applying \(\mu\) and \(\delta\) in Proposition \ref{reg-at-critical-points}, we consider parameters \(r_\star\), \(\lambda\), and \(\alpha\). 

For each nonnegative integer $i$, let
$$
    \mathcal{A}(i) \coloneqq \inf_{e \in \mathbb{S}^{n-1}}\left|\left\{x \in B_{r_\star\lambda^i}\colon Dv(x) \cdot e \leq \delta \lambda^{\alpha i}  \right\} \right|,
$$
and define
\begin{equation}\label{positivity every direction main result}
    \iota \coloneqq \inf\left\{i \in \mathbb{N} \colon \mathcal{A}(i) \leq \mu |B_{r_\star \lambda^i}| \right\}.
\end{equation}

We analyze the proof into two cases.

\textit{Case $\iota = \infty$}. From Proposition \ref{reg-at-critical-points}, 
$$
    \|Dv\|_{L^\infty\left(B_{r_\star\lambda^{i+1}}\right)} \leq \lambda^{\alpha (i+1)}, \quad \mbox{for each} \quad i \in \mathbb{N}.
$$
In particular, this implies that $Dv(0) = 0$. Additionally, for each $x \in B_{r_\star \lambda}$, there exists $j=j(x) \in \mathbb{N}$ such that $r_\star \lambda^{j+1} \leq |x| \leq r_\star \lambda^{j}$. Hence, we obtain
\begin{equation}
    |Dv(x)| \leq \lambda^{\alpha j} \leq (r_\star \lambda)^{-\alpha}|x|^\alpha.
\end{equation}
For $x \in B_2 \backslash B_{r_\star \lambda}$, we use that $\|Dv\|_{L^\infty(B_2)} \leq 1$, to get
$$
    |Dv(x)| \leq 1 \leq (r_\star \lambda)^{-\alpha}|x|^\alpha.
$$
Therefore, 
$$
    |Dv(x)| \leq C|x|^{\alpha} \quad \mbox{for} \quad x \in B_2.
$$

\textit{Case $\iota < \infty$}. Immediately, for some $e \in \mathbb{S}^{n-1}$, we have
\begin{equation}
    \left|\left\{x \in B_1\colon Dw(x) \cdot e \leq \delta \right\} \right| \leq \mu |B_1|,
\end{equation}
provided
$$
    w(x) \coloneqq \frac{v(r_\star\lambda^{\iota}x)}{r_\star\lambda^{\iota(1+\alpha)}}.
$$
According the proof of Proposition \ref{reg-at-critical-points}, we derive that $\|Dw\|_{L^\infty(B_2)} \leq 1$ and
$$
    |w(x)| \leq \max\left\{1,|x|^{1+\alpha_1} \right\}, \quad \mbox{for} \quad x \in \mathbb{R}^n.
$$
In the sequel, by taking $\varepsilon = \lambda_\star^{2}$, we apply Lemma \ref{proximity with affine functions} for $w$, obtaining so
$$
    \|w - \ell\|_{L^\infty(B_1)} \leq \lambda_\star^{2},
$$
for some affine function $\ell$ with $|D\ell| = 1$. From Proposition \ref{small pert argument}, we obtain
$$
    |Dw(x) - Dw(0)| \leq C|x|^{2s-1} \quad \mbox{for} \quad x \in B_{1/2},
$$
and so,
\begin{equation*}
|Dv(y) - Dv(0)| \leq C|y|^\alpha \quad \mbox{for} \quad y \in B_{ r_\star \lambda^{\iota}}.
\end{equation*}
We use the fact that \eqref{positivity every direction main result} holds up to index $\iota - 1$, yielding
$$
    \|Dv\|_{L^\infty\left(B_{r_\star \lambda^j}\right)} \leq \lambda^{\alpha j}, \quad \mbox{for} \quad j = 0, \cdots, \iota.
$$
Consequently, for each $y \in B_{r_\star} \backslash B_{r_\star \lambda^\iota}$, 
$$
    |Dv(y) - Dv(0)| \leq 2\|Dv\|_{L^\infty(B_{r_\star \lambda^j})} \leq 2 \lambda^{\alpha j} \leq C|y|^\alpha,
$$
for $j=j(y) \in \{0,1,\cdots, \iota - 1\}$, satisfying $r_\star \lambda^{j+1} \leq |y| \leq r_\star \lambda^{j}$.

This provides the desired estimate in $B_{r_\star}$. Since $Dv$ is normalized, we argue as before to extend the estimate for $Dv$ up to $B_2$. From this, we obtain
$$
    |Du(z) - Du(x_0)| \leq C\left( \|u\|_{L^\infty(\mathbb{R}^n)} + \|Du\|_{L^\infty(B_{3/4})}\right),
$$
 for each $z \in B_{1/2}(x_0)$.
\end{proof}

Let us briefly explain how to extend Theorem \ref{main theorem} to the case with nonhomogeneous right-hand side and bounded solutions.

\begin{proposition}\label{nonhomg case}
Let $u \in C(B_1) \cap L^\infty(\mathbb{R}^n)$ be a viscosity solution to \eqref{introprob}, for some $s\in (1/2,1)$. Then, $u$ is locally $C^{1,\alpha}$, for some universal $\alpha\in (0,1)$, depending only on $n$ and $s$. Furthermore, there exists $C = C(n,s)$, such that
$$
    \|u\|_{C^{1,\alpha}(B_{1/2})} \leq C\,\left(\|u\|_{L^\infty(\mathbb{R}^n)} + \|f\|_{C^{0,1}(B_1)}\right).
$$   
\end{proposition}

Following the program developed here, the first step is to adapt Lemma \ref{subsolution-lemma}. The proof will be the same, except that the difference quotient $w^h$ will solve
$$
    \Delta^s w^h = f_h \coloneqq\frac{f(x+h)-f(x)}{h}.
$$
Lipschitz continuity of the right-hand side plays a role in controlling the $L^\infty$ size of $f_h$.

\begin{lemma}
Let $\eta \colon \mathbb{R}^n \to [0,1]$ be a smooth cut-off function satisfying
$$
    \eta = 1 \quad \mbox{in} \quad  B_1, \quad \mbox{and} \quad \eta = 0 \quad \mbox{in} \quad \mathbb{R}^n \backslash B_2.
$$
If $u$ is a solution to \eqref{introprob} and $e \in \mathbb{S}^{n-1}$, then $v = \eta\,(\partial_e u - \mu)_+$ solves
$$
    \Delta^s v \geq -C\left(\|u\|_{L^1_s(\mathbb{R}^n)} +\|Df\|_{L^\infty(B_1)} \right)\quad \mbox{in }\; B_{1/2},
$$
for any $\mu \in (0,1)$ and $C$ is a dimensional constant.
\end{lemma}

The rest of the program now has to consider this new ingredient, which can be done by following the ideas developed in \cite{AST}.

\medskip

The proof of Theorem \ref{general theorem} now follows through a cut-off argument. It a consequence of Proposition \ref{nonhomg case}.

\begin{proof}[Proof of Theorem \ref{general theorem}]
Define $v \coloneqq u \chi_{B_1}$ and $w \coloneqq u (1 - \chi_{B_1})$. Since $u = v + w$, it follows that within $\Omega \cap B_1$
$$
    f = \Delta^s u = \Delta^s v + \Delta^s w.
$$
Thus, if we denote $g \coloneqq - \Delta^s w$ we get that the function $v$ solves \eqref{introprob}. Observe also that for points in $B_{3/4}$ we have
$$
    |Dg(x)| \leq (n+2s)\int_{\mathbb{R}^n\backslash B_1} |u(y)||y-x|^{-n-2s-1}dy \leq C\|u\|_{L^1_s(\mathbb{R}^n)},
$$
and we can apply Proposition \ref{nonhomg case} with $B_1$ replaced by $B_{3/4}$.
\end{proof}

\section{Fully Nonlinear Operators}\label{extensions-sct}

We briefly outline how our results extend to a broader class of equations. Specifically, we consider the class \( \mathcal{L}_1(s) \), first introduced in \cite{CS09}, which consists of kernels \( \mathcal{K} \colon \mathbb{R}^n \setminus \{0\} \to \mathbb{R} \) satisfying  
\[
    \Lambda^{-1} \leq \mathcal{K}(y)|y|^{n+2s} \leq \Lambda, \quad \text{and} \quad |D\mathcal{K}(y)| \leq \Lambda |y|^{-n-2s-1} 
\]
for all \( y \in \mathbb{R}^n \setminus \{0\} \). A nonlocal operator \( \mathcal{I} \) is said to be elliptic with respect to the class \( \mathcal{L}_1(s) \) if it satisfies the following inequality:  
\begin{equation}
	M^-_{\mathcal{L}_1(s)}[w](x) \leq \mathcal{I}[u+w](x) - \mathcal{I}[u](x) \leq M^+_{\mathcal{L}_1(s)}[w](x),
\end{equation}  
where the extremal operators are defined as  
\[
	M^-_{\mathcal{L}_1(s)}[w](x) \coloneqq \inf_{\mathcal{K} \in \mathcal{L}_1(s)}L_{\mathcal{K}}[w](x), \quad M^+_{\mathcal{L}_1(s)}[w](x) \coloneqq \sup_{\mathcal{K} \in \mathcal{L}_1(s)}L_{\mathcal{K}}[w](x),
\]  
with  
\[
    L_{\mathcal{K}}[w](x) \coloneqq \int_{\mathbb{R}^n} (w(y)-w(x)) \mathcal{K}(y-x) \, dy.
\]  
This definition characterizes nonlocal ellipticity in terms of the extremal influence of the class \( \mathcal{L}_1(s) \), playing a crucial role in the analysis to be developed in the following.

For this nonlocal operator $\mathcal{I}$, we consider solutions to
\begin{equation}\label{introprobgen}
\left\{
\begin{array}{rcl}
\mathcal{I} u = f & \text{in} & B_1 \cap \Omega, \\[0.1cm]
|Du| = 0 & \text{in} & B_1 \setminus \Omega.
\end{array}
\right.
\end{equation}
Although the strategy is analogous, the estimates will now depend on the ellipticity constant \( \Lambda \). To illustrate this, we state below the corresponding version of Lemma \ref{subsolution-lemma}.

\begin{lemma}
Let $\eta \colon \mathbb{R}^n \to [0,1]$ be a smooth cut-off function satisfying
$$
    \eta = 1 \quad \mbox{in} \quad  B_1, \quad \mbox{and} \quad \eta = 0 \quad \mbox{in} \quad \mathbb{R}^n \backslash B_2.
$$
If $u$ is a solution to \eqref{introprobgen} and $e \in \mathbb{S}^{n-1}$, then $v = \eta\,(\partial_e u - \mu)_+$ solves
$$
    \mathcal{M}_{\mathcal{L}_1(s)}v \geq -C\left(\|u\|_{L^1_s(\mathbb{R}^n)} +\|Df\|_{L^\infty(B_1)} \right)\quad \mbox{in }\; B_{1/2},
$$
for any $\mu \in (0,1)$ and $C$ depends on $n$, $s$ and $\Lambda$.
\end{lemma}

\begin{proof}
We can assume $\mathcal{I}[0] = 0$ for simplicity. We consider
$$
    w^h(x) \coloneqq \frac{u(x+h)-u(x)}{h} = w_1 + w_2,
$$
where $w_1 = w^h \chi_{B_1}$ and $w_2 = w^h(1- \chi_{B_1})$. By ellipticity assumption on $\mathcal{I}$, we have $\mathcal{M}_{\mathcal{L}_1(s)} w^h \geq f^h$. Therefore,
$$
    \mathcal{M}_{\mathcal{L}_1(s)} w_1 \geq - \mathcal{M}_{\mathcal{L}_1(s)} w_2 - \|Df\|_{L^\infty(B_1)}
$$
As before,
$$
    \left| - \mathcal{M}_{\mathcal{L}_1(s)} w_2 \right| \leq C(n,s,\Lambda)\|u\|_{L^1_s(\mathbb{R}^n)}.
$$
Here we are using that the kernels in the class $\mathcal{L}_1(s)$ satisfies 
$$|D\mathcal{K}(y)|\leq \Lambda |y|^{-n-2s-1}, \quad \mbox{for} \quad y \not = 0.
$$
\end{proof}

The next ingredients we require are the corresponding versions of Lemma \ref{reg-at-critical-points-step1} and Proposition \ref{small pert argument}. Although their proofs follow the same general strategy as before, a few minor adjustments are required, and we briefly comment on them below. 

First, we invoke the most general formulation of Lemma \ref{lemma:CS09}, stated in terms of the nonlocal Pucci operators (see \cite[Theorem 10.4]{CS09}). With this modification, the proof of Proposition \ref{reg-at-critical-points} carries over unchanged. 

For the proof of the analogue of Proposition \ref{small pert argument}, only a minor modification is required: in estimate \eqref{eq:estimate}, the exponent $2s-1$ is replaced by a universal exponent $\overline{\alpha}\leq 2s-1$. This exponent arises from the interior regularity theory established in \cite{CS11}.

After these adjustments, we follow the same strategy as in the proof of Theorem \ref{main theorem}, to establish the following result.

\begin{theorem}\label{L_1 case}
Let $u$ be a viscosity solution to \eqref{introprobgen}, for some $s\in (1/2,1)$. Then, $u$ is locally $C^{1,\alpha}$, for some universal $\alpha\in (0,1)$, depending only on $n$, $s$ and $\Lambda$. Furthermore, there exists $C = C(n,s,\Lambda)$, such that
$$
    \|u\|_{C^{1,\alpha}(B_{1/2})} \leq C\,\left(\|u\|_{L^\infty(B_1)} + \|u\|_{L^1_s(\mathbb{R}^n)} + \|f\|_{C^{0,1}(B_1)}\right).
$$   
\end{theorem}

\medskip

{\small \noindent{\bf Acknowledgments.} DJA is supported by Conselho Nacional de Desenvolvimento Científico e Tecnológico (CNPq) grants 310020/2022-0 and 420014/2023-3. AS is supported by the King Abdullah University of Science and Technology (KAUST).}


\begin{thebibliography}{99}

\bibitem{ARO20} N. Abatangelo and X. Ros-Oton, \textit{Obstacle problems for integro-differential operators: Higher regularity of free boundaries}, Adv. Math. \textbf{360} (2020), 106931.

\bibitem{APT} D.J. Araújo, D. Prazeres and E. Topp, \textit{On fractional quasilinear equations with elliptic degeneracy}, Preprint, arXiv:2306.15452 (2023).

\bibitem{AST} D.J. Araújo, A.O. Sobral and E.V. Teixeira, 
\textit{Regularity in diffusion problems with gradient activation}, J. Funct. Anal. 289 (2025), no. 11, 111147. \href{https://doi.org/10.1016/j.jfa.2025.111147}{https://doi.org/10.1016/j.jfa.2025.111147}

\bibitem{BCI08} G. Barles, E. Chasseigne and C. Imbert, \textit{On the Dirichlet Problem for Second Order Elliptic Integro-Differential Equations}, Indiana Univ. Math. J. \textbf{57}(2008), no. 1, 213-246.

\bibitem{BI08} G. Barles and C. Imbert, \textit{Second-order Eliptic Integro-Differential Equations: Viscosity Solutions' Theory Revisited}, IHP Anal. Non Linéare, Vol. \textbf{25}(2008) no. 3, 567-585.

\bibitem{BFRO18} B. Barrios, A. Figalli and X. Ros-Oton, \textit{Global Regularity for the Free Boundary in the Obstacle Problem for the fractional Laplacian}, Amer. J. of Math., Vol. \textbf{40}(2018), 415--447.

\bibitem{BR99} J. Berger and  J. Rubinstein
{\it On the zero set of the wave function in superconductivity.}   Comm. Math. Phys. {\bf 202} (1999), no. 3, 621--628.

\bibitem{CROS17} L. Caffarelli, X. Ros-Oton, and J. Serra, \textit{Obstacle problems for integro-differential operators: regularity of solutions and free boundaries}. Invent. math. 208, 1155--1211 (2017).

\bibitem{CS02} L. Caffarelli and J. Salazar, \textit{Solutions of fully nonlinear elliptic equations with patches of zero gradient: existence, regularity and convexity of level curves}, Trans. Amer. Math. Soc. {\bf 354} (2002), no.~8, 3095--3115.

\bibitem{CSS04} L. Caffarelli, J. Salazar and H. Shahgholian, \textit{Free-boundary regularity for a problem arising in superconductivity}, Arch. Ration. Mech. Anal. {\bf 171} (2004), no.~1, 115--128.

\bibitem{CS09} L. Caffarelli and L. Silvestre,
\textit{Regularity theory for fully nonlinear integro-differential equations}, Comm. Pure Appl. Math. {\bf 62} (2009), no. 5, 597-638.

\bibitem{CS11} L. Caffarelli and L. Silvestre,
\textit{Regularity results for nonlocal equations by approximation}, Arch. Ration. Mech. Anal. {\bf 200}(2011), no. 1, 59-88.

\bibitem{C95} S. Chapman, {\it A mean-field model of superconducting vortices in three dimensions.}
SIAM J. Appl. Math. {\bf 55} (1995), no. 5, 1259--1274.

\bibitem{CIL92} M. G. Crandall, H. Ishii and P.-L. Lions, \textit{User’s guide to viscosity solutions of second order partial differential equations}, Bull. Am. Math. Soc. {\bf 27}(1992), no. 1, 1-67.

\bibitem{ESS} C.~M. Elliott, R.~M. Sch\"atzle and B.~E.~E. Stoth, \textit{Viscosity solutions of a degenerate parabolic-elliptic system arising in the mean-field theory of superconductivity}, Arch. Ration. Mech. Anal. {\bf 145} (1998), no.~2, 99--127,

\bibitem{DPV} E. Di~Nezza, G. Palatucci and E. Valdinoci, \textit{Hitchhiker's guide to the fractional Sobolev spaces}, Bull. Sci. Math. {\bf 136} (2012), no.~5, 521--573.

\bibitem{FS} A. Figalli and H. Shahgholian, \textit{A general class of free boundary problems for fully nonlinear elliptic equations}, Arch. Ration. Mech. Anal. {\bf 213} (2014), no. 1, 269--286.

\bibitem{IS13} C. Imbert and L. Silvestre, \textit{$C^{1,\alpha}$ regularity of solutions of some degenerate fully non-linear elliptic equations}, Adv. Math {\bf 233} (2013), no. 1, 196-206.

\bibitem{IJS} C. Imbert, T. Jin and L. Silvestre, \textit{H\"older gradient estimates for a class of singular or degenerate parabolic equations}, Adv. Nonlinear Anal. {\bf 8} (2019), no.~1, 845--867.

\bibitem{JN17} Y. Jhaveri and R. Neumayer, \textit{Higher regularity of the free boundary in the obstacle problem for the fractional Laplacian}, Adv. in Math. Vol. 311, 2017, 748--795.

\bibitem{JS} T. Jin and L. Silvestre, \textit{H\"older gradient estimates for parabolic homogeneous $p$-Laplacian equations}, J. Math. Pures Appl. (9) {\bf 108} (2017), no.~1, 63--87.

\bibitem{PT} D. dos~Prazeres and E. Topp, \textit{Interior regularity results for fractional elliptic equations that degenerate with the gradient}, J. Differential Equations {\bf 300} (2021), 814--829.

\bibitem{RO16} X. Ros-Oton, \textit{Nonlocal Elliptic Equations in Bounded Domains: a Survey}. Publ. Mat., 60 (2016), 3--26.

\bibitem{ROS17} X. Ros-Oton, J. Serra, \textit{Boundary regularity estimates for nonlocal elliptic equations in 
$C^1$ and $C^{1,\alpha}$ domains}. Annali di Matematica 196, 1637--1668 (2017).

\bibitem{SILV06} L. Silvestre, \textit{H\"older estimates for solutions of integro-differential equations like the fractional Laplace}, Indiana Univ. Math. J. {\bf 55} (2006), no.~3, 1155--1174.

\end{thebibliography}
\end{document}